\documentclass[12pt]{amsart}

\usepackage{amscd,amssymb,amsopn,amsmath,amsthm,mathrsfs,graphics,amsfonts,enumerate,verbatim,calc
}
\usepackage{amscd}

\usepackage{color}

\usepackage[OT2,OT1]{fontenc}
\newcommand\cyr{%
\renewcommand\rmdefault{wncyr}%
\renewcommand\sfdefault{wncyss}%
\renewcommand\encodingdefault{OT2}%
\normalfont
\selectfont}
\DeclareTextFontCommand{\textcyr}{\cyr}

\usepackage{amssymb,amsmath}
\usepackage{caption}

\DeclareFontFamily{OT1}{rsfs}{}
\DeclareFontShape{OT1}{rsfs}{n}{it}{<-> rsfs10}{}
\DeclareMathAlphabet{\mathscr}{OT1}{rsfs}{n}{it}

\numberwithin{equation}{section}
\newtheorem{theorem}{Theorem}[section]
\newtheorem{lemma}[theorem]{Lemma}

\newtheorem{corollary}[theorem]{Corollary}

\theoremstyle{definition}
\newtheorem{definition}[theorem]{Definition}
\newtheorem{remark}[theorem]{Remark}
\theoremstyle{remark}

\renewcommand{\ker}{\operatorname{Ker}}

\newcommand{\Hom}{\operatorname{Hom}}

\newcommand{\depth}{\operatorname{depth}}

\newcommand{\fm}{\frak{m}}
\newcommand{\fp}{\frak{p}}
\newcommand{\fq}{\frak{q}}

\begin{document}
\title[Frobenius test exponents]{A sharp bound for the Frobenius test exponents in generalized Cohen-Macaulay local rings}

\author[Duong Thi Huong]{Duong Thi Huong}
\address{Department of Mathematics, Thang Long University, Hanoi, Vietnam}
\email{huongdt@thanglong.edu.vn}

\author[Pham Hung Quy]{Pham Hung Quy}
\address{Department of Mathematics, FPT University, Hanoi, Vietnam}
\email{quyph@fe.edu.vn}



\maketitle

\begin{abstract} Let $(R,\fm)$ be a generalized Cohen-Macaulay local ring of prime characteristic $p$. In this paper we give a sharp bound for the Frobenius test exponent of parameter ideals. Namely, we prove that
$$\mathrm{Fte}(R) \le \lceil  \log_p(2n_0)\rceil + \mathrm{HSL}(R),$$
where $n_0$ is the integer such that $\frak m^{n_0} \, H^i_{\frak m}(R) = 0$ for all $i < \mathrm{dim}(R)$, and $\lceil x\rceil$ is the smallest integer that is greater than or equal to $x$.      
\end{abstract}

\section{Introduction}
Let $(R,\fm)$ be a Noetherian local ring of prime characteristic $p$, $M$ be a finitely generated $R$-module of dimension $d>0$, $I$ an ideal of $R$. Let $F: R \to R; x \mapsto x^p$ be Frobenius endomorphism of the ring. The nilpotent part under the Frobenius endomorphism modulo $I$ is called the {\it Frobenius closure} of $I$, defined by $I^F = \{x \mid x^{p^e} \in I^{[p^e]} \text{ for some } e \ge 0\}$, where $I^{[p^e]} = (x^{p^e} \mid x \in I)$ is the {\it Frobenius closure} of $I$. By the Noetherianness of $R$ there is an integer $e$, depending on $I$, such that $(I^F)^{[p^e]} = I^{[p^e]}$. The {\it Frobenius test exponent} of $I$ is defined by $\mathrm{Fte}(I) = \min \{e \mid (I^F)^{[p^e]} = I^{[p^e]}\}$. Under mild conditions, $R$ is $F$-pure if and only if $\mathrm{Fte}(I) = 0$ for all $I$ by Hochster \cite{H77}. In general, we cannot expect a uniform bound for all $\mathrm{Fte}(I)$ by Brenner \cite{B06}. Restricting to the class of parameter ideals, for any local ring $(R, \fm)$, we define the {\it Frobenius test exponent for parameter ideals} $\mathrm{Fte}(R)$ as follows
$$\mathrm{Fte}(R) = \min \{e \mid (\fq^F)^{[p^e]} = \fq^{[p^e]} \, \text{for all parameter ideals }\, \fq\},$$
 and $\mathrm{Fte}(R) = \infty$ if we have no such integer. Katzman and Sharp asked whether $\mathrm{Fte}(R) < \infty$ for any (equidimensional) local ring (i.e. there exists an integer $C$ such that  $\mathrm{Fte}(R) \leq C$).

 There are several classes of rings for which $\mathrm{Fte}(R)$ is finite, such as Cohen-Macaulay rings (cf. \cite{KS06}), generalized Cohen-Macaulay rings (cf. \cite{HKSY06}), $F$-nilpotent rings (i.e. $H^i_{\fm}(R) = 0^F_{H^i_{\fm}(R)}$ for all $i < \mathrm{dim}(R)$) (cf. \cite{Q19}) and generalized weakly $F$-nilpotent rings, i.e. $H^i_{\fm}(R) / 0^F_{H^i_{\fm}(R)}$ has finite length for all $i < \mathrm{dim}(R)$ in \cite{M19}. In general, for any local ring $R$, we have $\mathrm{Fte}(R) \geq \mathrm{HSL}(R)$ in \cite{HQ19}. 

In another context, we consider the question of whether there exists a uniform bound for the Frobenius test exponent of certain classes of ideals. When $R$ has positive finiteness dimension $t$, we proved that \cite[Theorem 3.6]{HQ22} there exists an integer $C$ such that for any filter regular sequence $x_1,\ldots,x_s$ of length at most $t$ we have $\mathrm{Fte}(x_1, \ldots, x_s) \leq C$. The number $C$ is related to the invariant $\mathrm{HSL}(R)$, which is always finite and is defined by $H^i_{\mathfrak m}(R)$ for $i=1,\ldots ,\mathrm{dim}(R)$ (see the next section for details). 

It should be noted that the Cohen-Macaulay rings have $\mathrm{Fte}(R)=\mathrm{HSL}(R)$, while in generalized Cohen-Macaulay rings the second author \cite{Q19} proved that
$$\mathrm{Fte}(R) \le \lceil  \log_p(2n_0)\rceil \; +\;\sum_{i = 0}^{\mathrm{dim}(R)} \binom{\mathrm{dim}(R)}{i} \mathrm{HSL} (H^i_{\fm}(R)),$$
where $n_0$ is the integer such that $\frak m^{n_0} \, H^i_{\frak m}(R) = 0$ for all $i < \mathrm{dim}(R)$, and $\lceil x\rceil$ is the smallest integer that is greater than or equal to $x$.
In \cite[Theorem 1.1]{Ma15}, it was proved that if $R$ is a generalized Cohen-Macaulay ring then we have $\mathrm{Fte}(R)=0$ if and only if $\mathrm{HSL}(R)=0$. In this case, $R$ becomes a Buchsbaum ring. 
Let $R$ be an excellent generalized Cohen-Macaulay ring that is $F$-injective on the punctured spectrum. Let $n_0$ be a positive integer such that $\fm^{n_0} H^i_{\fm}(R) = 0$ for all $i < \mathrm{dim}(R)$, and $\fm^{n_0} 0^F_{H^{\mathrm{dim}(R)}_{\fm}(R)} = 0$. We have
$$\mathrm{Fte}(R) \le \lceil  \log_p(2n_0)\rceil + \mathrm{HSL}(R),$$
where $\lceil x\rceil$ is the smallest integer that is greater than or equal to $x$ \cite[Corollary 3.6]{HQ23}.
The condition that $R$ is $F$-injective in the punctured spectrum guarantees that $0^F_{H_{\fm}^ {\mathrm{dim}(R)}(R)}$ has finite length. We wonder whether this condition can be omitted in order to consider the problem for generalized Cohen--Macaulay rings \cite[Question 3.7]{HQ23}.

The aim of this paper is to provide a better upper bound for the Frobenius test exponent of all ideals generated by filter sequences of length at most the finiteness dimension of the ring. As a consequence, we obtain a sharp bound for $\mathrm{Fte}(R)$ when $R$ is generalized Cohen-Macaulay.

The main results of this paper are the following
\begin{theorem}
 Let $(R,\fm)$ be a local ring of prime characteristic $p$ and of positive finiteness dimension $t>0$ . Let $n_0$ be a positive integer such that $\fm^{n_0} H^i_{\fm}(R) = 0$ for all $i < t$. Then for every filter regular sequence $x_1,\ldots,x_t$ of $R$ we have
$$\mathrm{Fte}(x_1,\ldots,x_t) \le \lceil  \log_p(2n_0)\rceil + \mathrm{HSL}(R),$$
where $\lceil x\rceil$ is the smallest integer that is greater than or equal to $x$. 
\end{theorem}
\begin{corollary}
 Let $(R,\fm)$ be a generalized Cohen-Macaulay local ring of prime characteristic $p$. Let $n_0$ be a positive integer such that $\fm^{n_0} H^i_{\fm}(R) = 0$ for all $i < \mathrm{dim}(R)$. Then 
$$\mathrm{Fte}(R) \le \lceil  \log_p(2n_0)\rceil + \mathrm{HSL}(R),$$
where $\lceil x\rceil$ is the smallest integer that is greater than or equal to $x$.     
\end{corollary}
In the next section, we recall the basic notions and relevant materials. We prove the main results in the last section.
\section{Preliminaries}
\subsection{Standard sequences}
 In this subsection, the ring $(R,\fm)$ is not necessarily of prime characteristic. $M$ is a finitely generated $R$-module and $x_1,\ldots, x_s$ is a sequence of elements of $R$.
We recall the notion of a filter regular sequence, introduced by Cuong, Schenzel, and Trung in \cite{CST78}, which generalizes the notion of a regular sequence. The sequence $x_1,\ldots, x_s$ is a {\it filter regular sequence} on $M$ if for every $1\le i\le s$ we have
$$(x_1,\ldots,x_{i-1})M:_Mx_i \subseteq \bigcup_{n\geq 1}(x_1,\ldots,x_{i-1})M:_M{\fm^{n}}.$$
If the sequence $x_1,\ldots, x_s$ is a filter regular sequence on $M$ then the sequence $x_1^{n_1},\ldots,x_s^{n_s}$ is a filter regular sequence on $M$ for all $n_1,\ldots,n_s\geq 1$.
\begin{definition}
Let $(R,\fm)$ be a Noetherian local ring, $M$ a finitely generated $R$-module of dimension $d>0$, and $\fq=(x_1,\ldots, x_d)$ a parameter ideal of $M$. Then
\begin{enumerate}
\item $M$ is called {\it generalized Cohen-Macaulay} if there exists a positive integer $n_0$ such that $\mathfrak m^{n_0}H^i_{\mathfrak m}(M)=0$ for all $i<d$.
\item The parameter ideal $\fq$ is called {\it standard} if for all $i+j<d$ we have
$$\fq H^j_{\fm}(M/(x_1,\ldots,x_i)M)=0.$$ 
\end{enumerate}
\end{definition}
\begin{remark}
Let $(R,\fm)$ be a Noetherian local ring, $M$ a finitely generated $R$-module of dimension $d>0$ and $\fq=(x_1,\ldots,x_d)$ a parameter ideal of $M$. Set $\fq_i=(x_1,\ldots,x_i)$ for $i=1,\ldots,d$. We consider the following assertions
\begin{enumerate}
\item $M$ is  generalized Cohen-Macaulay.
\item $I(M) := \sup \;\{\ell(M/\mathfrak qM)-e(\mathfrak q;M)\} <\infty,$
where $\fq$ runs through all parameter ideals of $M$.
\item There exists a positive integer $n$ such that
$$\mathfrak q_{i-1}M :_M x_i \subseteq \mathfrak q_{i-1}M :_M \mathfrak m^n$$
for every system of parameters of $M$ and $i=1,\ldots,d$.
\item There exists a standard system of parameters $x'_1,\ldots,x'_d$ of M.  In this case, we have $I(M) = I(\fq', M)$ where $\fq' =(x'_1,\ldots,x'_d)$.
\item Every system of parameters $x_1,\ldots,x_t$ of $M$ is a filter regular sequence on $M$. 
\item $M$ is equidimensional and $M_{\fp}$ are Cohen-Macaulay for all $\fp \in \mathrm{Spec}(R)\setminus \{\fm\}$.
\end{enumerate}
Then the first four assertions are equivalent, and they are also equivalent to the last two provided that R is a homomorphic image of a Cohen–Macaulay local ring.  
\end{remark}
\begin{definition}
Let $M$ be a finitely generated module over a local ring $(R,\mathfrak m)$.
The \it{finiteness dimension} $f_{\mathfrak m}(M)$ of $M$ with respect to
$\mathfrak m$ is defined as follows:
\[
f_{\mathfrak m}(M):=
\inf\{\, i \mid H^i_{\mathfrak m}(M)
\text{ is not finitely generated}\,\}
.
\]
For convenience, we adopt the convention that $f_{\mathfrak m}(M) = \infty$ when $\mathrm{dim}(M)=0$ or $\mathrm{dim}(M)=-\infty$.
\end{definition}
\begin{definition}
Let $M$ be a finitely generated module over a local ring $(R,\mathfrak m)$
such that \(t=f_{\mathfrak m}(M)<\infty\), and let
\(x_1,\ldots,x_s\), \(s\le t\), be a filter regular sequence on \(M\).
Then we say that \(x_1,\ldots,x_s\) is \it{standard sequence} of \(M\) if
\[
(x_1,\ldots,x_s)\,
H^i_{\mathfrak m}\!\left(
\frac{M}{(x_1,\ldots,x_j)M}
\right)=0
\]
for all \(i+j<s\).
\end{definition}
\begin{remark}
\begin{enumerate}
    \item Let $M$ be a finitely generated module with $\mathrm{dim}(M)=d>0$. Then $0<f_{\mathfrak m}(M)\leq d$ since $H^d_{\fm}(M)$ is never finitely generated. Moreover, $f_{\mathfrak m}(M)= d$ if and only if $M$ is generalized Cohen-Macaulay.
    \item  Assume that \(f_{\mathfrak m}(M)<\infty\). By Grothendieck's finiteness theorem, we have
\[
f_{\mathfrak m}(M)=
\min\{
\depth(M_{\mathfrak p})+\dim R/\mathfrak p
\mid
\mathfrak p\neq \mathfrak m
\},
\]
provided that \(R\) is a homomorphic image of a regular local ring \cite[Theorem 9.5.2]{BS98}.
\item If \(x_1,\ldots,x_s\) with \(s\le f_{\mathfrak m}(M)\) is a filter regular
sequence on \(M\), then it is a filter regular sequence on \(M\) in any order (cf. \cite[Lemma 2.9]{Q13}).
\end{enumerate}
\end{remark}
We need also splitting results for local cohomology \cite[Corollary 4.11 and Corollary 4.12]{CQ11}.
\begin{theorem}\label{splitting1} Let $(R,\mathfrak m)$ be a Noetherian local ring, $M$ be a finitely generated module $R$-module of dimension $d>0$ and of positive finiteness dimension $t$. Let $n_0$ be a positive integer such that $\fm^{n_0}H^i_{\fm}(M) = 0$ for all $i < t$. Then for any filter regular element $x \in \fm^{2n_0}$ of $M$, we have
$$H^i_{\fm}(M/xM) \cong H^i_{\fm}(M) \oplus H^{i+1}_{\fm}(M)$$
for all $i < t-1$, and 
$$0:_{H^{d-1}_{\fm}(M/xM)} \fm^{n_0} \cong H^{d-1}_{\fm}(M) \oplus 0:_{H^{d}_{\fm}(M)} \fm^{n_0}.$$ Moreover, if a filter regular sequence $x_1, \ldots, x_d$ is contained in $\fm^{2n_0}$, then it is standard.
\end{theorem}
\begin{theorem}\label{splitting2}
Let $(R,\mathfrak m)$ be a Noetherian local ring, $M$ be a finitely generated module $R$-module of positive finiteness dimension $t$. Let $n_0$ be a positive integer such that $\fm^{n_0}H^i_{\fm}(M) = 0$ for all $i < t$. 
Let \(x_1,\ldots,x_t\) be a filter regular sequence
of \(M\) contained in \(\mathfrak m^{2n_0}\).
Then for all positive integers \(k\le n_0\) and all
\(j=1,\ldots,t\),
\[
\Hom_R(R/\mathfrak m^k,M/(x_1,\ldots,x_j)M)
\]
are independent of the choice of the sequence
\(x_1,\ldots,x_j\).
Moreover, we have
\[
\Hom_R(R/\mathfrak m^k,M/(x_1,\ldots,x_j)M)
\cong
\bigoplus_{i=0}^{j}
\Hom_R(R/\mathfrak m^k,H^i_{\mathfrak m}(M))^{\binom{j}{i}}.
\]
\end{theorem}
\begin{corollary}\label{qlim-fil reg seq}
Let $M$ be a finitely generated module of positive finiteness dimension $t$. Let $n_0$ be a positive integer such that $\mathfrak m^{n_0}H^i_{\mathfrak m}(M)=0
\quad \text{for all } i<t$. Then every filter regular sequence $(x_1,\ldots,x_t)$ of \(M\) contained in \(\mathfrak m^{2n_0}\) is standard. As a consequence, it is an $\mathfrak m^{n_0}$-weak $M$-sequence (i.e. $(x_1,\ldots,x_{i-1})M:_Mx_i\subseteq (x_1,\ldots,x_{i-1})M:_M\mathfrak m^{n_0} $ for all $i=1,\ldots ,t$).
\end{corollary}
\begin{proof}
Take any filter regular sequence $(x_1,\ldots,x_t)$ of \(M\) contained in \(\mathfrak m^{2n_0}\). By Theorem \ref{splitting1},  $(x_1,\ldots,x_t)$ is standard and 
$$H^0_{\mathfrak m}(M/(x_1,\ldots, x_{i-1})M)\cong \bigoplus_{s=0}^{i-1} H^s_{\mathfrak{m}}(M)^{\binom{i-1}{s}}.$$
So $(x_1,\ldots, x_t)H^0_{\mathfrak m}(M/(x_1,\ldots, x_{i-1})M)=0$ and we have
$$(x_1,\ldots, x_{i-1})M:_Mx_i = H^0_{\mathfrak m}(M/(x_1,\ldots, x_{i-1})M) \subseteq (x_1,\ldots, x_{i-1})M:_M\mathfrak m^{n_0}.$$
The first equality is due to the definition of the filter regular sequence $(x_1,\ldots,x_t)$ on \(M\), and the second inclusion follows from the fact that $m^{n_0}H^s_{\mathfrak{m}}(M)=0$ for all $s=0,\ldots t-1$. Hence, $(x_1,\ldots, x_{t})$ is $\mathfrak m^{n_0}$-weak $M$-sequence.
\end{proof}

\subsection{The limit closure}

Let $x_1,\ldots,x_t$ be a sequence in $R$. There is a useful way to describe the top local cohomology. It can be given as the direct limit
of Koszul cohomologies
\[
H^t_I(M) \cong \varinjlim M/(x_1^n,\ldots,x_t^n)M,
\]
with the map in the system
$ \varphi_{n,m} :
M/(x_1^n,\ldots,x_t^n)M
\longrightarrow
M/(x_1^m,\ldots,x_t^m)M$ is multiplication by $ (x_1\cdots x_t)^{m-n}$ for all \(m \geq n\). The kernel of the map is $(x_1,\ldots,x_t)^{\mathrm{lim}}_M/(x_1,\ldots,x_t)M$ where 
$$(x_1,\ldots,x_t)^{\mathrm{lim}}_M := \bigcup_{n\ge 1} \left( (x_1^{n+1}, \ldots, x_t^{n+1})M :_M (x_1 \cdots x_t)^n \right)$$
is called the \textit{limit closure} of $(x_1,\ldots,x_t)$ in $M$ and it does not depend on the choice of its generators $x_1,\ldots, x_t$ of ideal $(x_1,\ldots, x_t)$. It contains a lot of information of the ring. $M$ is a Cohen-Macaulay module iff $\mathfrak q^{\mathrm{lim}}_M =\mathfrak q$ for some (all) parameter ideals $\mathfrak q$ of $M$. If $M$ is unmixed, then $M$ is generalized Cohen-Macaulay iff 
    $$\sup \{e(\mathfrak q,M)-\ell(M/\mathfrak q^\mathrm{lim}_M)\}<\infty $$ 
    for all parameter ideals $\mathfrak q$ of $M$ (cf. \cite[Corollary 3.3]{CN03}). Moreover, $M$ is Buchsbaum iff $\ell(\mathfrak q^\mathrm{lim}_M/\frak qM)$ is a constant for all parameter ideals $\mathfrak q$ of $M$, see \cite[Theorem 1.1]{MQ22}. By \cite[Theorem 5.1]{CHL99}, if $M$ is a generalized Cohen-Macaulay module, then for every system of parameters $x_1, \ldots, x_d$ of $M$, we have 
 \[
\ell_R\!\left(
(x_1,\ldots,x_{d})^{\mathrm{lim}}_M /(x_1,\ldots,x_{d})M \right) \leq \sum_{i=0}^{d-1}\binom{d}{i}\,\ell_R\!\left(H_{\mathfrak m}^i(M)\right).\]
Moreover, the equality occurs if $x_1, \ldots, x_d$ is standard.
    
    For a Buchsbaum module $M$ (cf. \cite[Theorem 4.7]{G83}), we have a beautiful form 
    $$\mathfrak q^{\mathrm{lim}}_M = \mathfrak q M  + \sum_{i=1}^d(x_1,\ldots, \widehat{x_i},\ldots,x_d)M:_Mx_i.$$ 
    This result is true for standard parameter ideals by \cite[Lemma 3.5]{T86}. We state it in the following. 
   \begin{lemma}\label{lemma qlim}
     Let $M$ be an arbitrary finitely generated module of positive dimension $d$ and of the finiteness dimension $t$. Then for every standard sequence $x_1,\ldots,x_t$ of \(M \) we have 
    $$(x_1,\ldots,x_s)^{\mathrm{lim}}_M = (x_1, \ldots, x_s) M  + \sum_{i=1}^s(x_1,\ldots, \widehat{x_i},\ldots,x_s)M:_Mx_i,$$
    for all $s \le t$. Moreover, 
    \[
\ell_R\!\left(
(x_1,\ldots,x_{s})^{\mathrm{lim}}_M /(x_1,\ldots,x_{s})M \right) = \sum_{i=0}^{s-1}\binom{s}{i}\,\ell_R\!\left(H_{\mathfrak m}^i(M)\right).\]

\end{lemma}
If $n_0$ is a positive integer such that $\mathfrak m^{n_0}H^i_{\mathfrak m}(M)=0$ for all $i<t$. For every filter regular sequence $x_1,\ldots,x_t$ of \(M\) contained in \(\mathfrak m^{2n_0}\) and for every $s\leq t$ we have 
    $$(x_1,\ldots,x_s)^{\mathrm{lim}}_M = (x_1, \ldots, x_s) M  + \sum_{i=1}^s(x_1,\ldots, \widehat{x_i},\ldots,x_s)M:_Mx_i \subseteq (x_1,\ldots,x_s)M:_M{\mathfrak m^{n_0}}.$$
\subsection{The Frobenius action on local cohomology}
In this subsection, let $R$ be a Noetherian ring containing a field of prime characteristic $p$. Let $F:R \to R, x \mapsto x^p$ denote the Frobenius endomorphism. 
\begin{definition} Let $I$ be an ideal of $R$, we define
\begin{enumerate}
\item The {\it $e$-th Frobenius power} of $I$ is $I^{[p^e]} = (x^{p^e} \mid x \in I)$.
  \item The {\it Frobenius closure} of $I$, $I^F = \{x \mid  x^{p^e} \in I^{[p^e]} \text{ for some } e \ge 0\}$.
\end{enumerate}
\end{definition}
\begin{definition}
Let $I$ be an ideal of $R$. By the Noetherianness of $R$ there is an integer $e$, depending on $I$, such that $(I^F)^{[p^e]}=I^{[p^e]}$. The smallest number $e$ satisfying the condition is called the {\it Frobenius test exponent of $I$}, and denoted by $\mathrm{Fte}(I)$,
$$\mathrm{Fte}(I)=\min \{e\mid (I^F)^{[p^e]}=I^{[p^e]}\}.$$
\end{definition}
A problem of Katzman and Sharp \cite[Introduction]{KS06} asks in its strongest form: does there exist an integer $e$, depending only on the ring $R$, such that for every ideal $I$ we have $(I^F)^{[p^e]} = I^{[p^e]}$. A positive answer to
this question, together with the actual knowledge of a bound for $e$, would give an algorithm to compute the Frobenius closure $I^F$. Unfortunately, Brenner \cite{B06} gave two-dimensional normal standard graded domains with no Frobenius test exponent. In contrast, Katzman and Sharp showed the existence of Frobenius test exponent if we restrict to the class of parameter ideals in a Cohen-Macaulay ring. Therefore, it is natural to ask the question whether there exists an integer $e$ such that for every parameter ideal $\fq$ of $R$ we have $(\fq^F)^{[p^e]} = \fq^{[p^e]}$. We define the {\it Frobenius test exponent for parameter ideals} of $R$, $\mathrm{Fte}(R)$, the smallest integer $e$ satisfying the above condition and $\mathrm{Fte}(R) = \infty$ if no such integer $e$ exists. It should be noted that the authors used the finiteness of $\mathrm{Fte}(R)$, if it exists, to find an upper bound of the multiplicity of a local ring \cite{HQ20}.\\
For any ideal $I = (x_1, \ldots, x_t)$, the Frobenius endomorphism $F:R \to R$ and its localizations induce a natural Frobenius action on local cohomology $F:H^i_I(R) \to H^i_{I^{[p]}}(R) \cong H^i_{I}(R)$ for all $i \ge 0$. In general, let $A$ be an Artinian $R$-module with Frobenius action $F: A \to A$. The {\it Frobenius closure $0^F_A$ of the zero submodule} of $A$ is defined to be the submodule of $A$ consisting of all elements $z$ such that $F^e(z) = 0$ for some $e \ge 0$. Hence $0^F_A$ is the nilpotent part of $A$ under the Frobenius action. By \cite[Proposition 1.11]{HS77}, \cite[Proposition 4.4]{L97} and \cite{Sh07} we have the following.
\begin{theorem}
Let $(R, \fm)$ be a local ring of prime characteristic $p>0$, and $A$ be an Artinian $R$-module with a Frobenius action $F: A \to A$. Then there exists a non-negative integer $e$ such that $0^F_A = \ker (A \overset{F^e}{\longrightarrow} A)$.
\end{theorem}
\begin{definition}
\begin{enumerate}
\item Let $A$ be an Artinian $R$-module with Frobenius action $F$.  The {\it Hartshorne-Speiser-Lyubeznik number} of $A$ is denoted by $\mathrm{HSL}(A)$ and is defined as
$$\mathrm{HSL}(A)=\min\{e\mid 0^F_A = \ker (A \overset{F^e}{\longrightarrow} A)\}.$$
\item Notice that $H^i_{\fm}(R)$ is always Artinian for all $i \ge 0$. We define the {\it Hartshorne-Speiser-Lyubeznik number} of a local ring $(R, \frak m)$ as follows
$$\mathrm{HSL}(R): = \min \{ e \mid   0^F_{H^i_{\fm}(R)} =   \ker (H^i_{\fm}(R) \overset{F^e}{\longrightarrow} H^i_{\fm}(R)) \text{ for all } i = 0, \ldots, d\}.$$
\end{enumerate}
\end{definition}

\section{Main result}
We first combine the splitting results for local cohomology, Theorems 2.7, 2.8, with the theory of limit closure. The following result was proved in the Buchsbaum case using different methods, those methods cannot be extended to the generalized Cohen–Macaulay case (see \cite[Lemma 2.6]{GMMQS26}).  
\begin{lemma}\label{qlim splitting} 
Let $(R,\fm)$ be a local ring and $M$ be a finitely generated $R$-module of positive finiteness dimension $t>0$. Let $n_0$ be a positive integer such that $\fm^{n_0}H_{\fm}^j(M)=0$ for all $j<t$. Then for every filter regular sequence $x_1,\ldots,x_t $ of $M$ contained in $\fm^{2n_0} $ , we have
$$ (x_1,\ldots,x_t)_M^{\mathrm{lim}}/(x_1,\ldots,x_t)M \cong \bigoplus_{i=0}^{t-1} H^i_{\mathfrak{m}}(M)^{\binom{t}{i}}.$$  
\end{lemma}
\begin{proof}
By Theorem \ref{splitting2}, 
$$((x_1,\ldots,x_t)M:_M \fm^{n_0})/(x_1,\ldots,x_t)M \cong \big( \bigoplus_{i=0}^{t-1} H^i_{\mathfrak{m}}(M)^{\binom{t}{i}} \big) \,\, \oplus\,\, 0:_{H^t_{\fm}(M)}\fm^{n_0}.$$
By \cite{NS95}, $H^t_{\fm}(M) \cong H^0_{\fm}(H^t_{(x_1,\ldots,x_t)}(M))$ is a submodule of $H^t_{(x_1,\ldots,x_t)}(M)$. Moreover, we have
$$(x_1,\ldots, x_t)^{\mathrm{lim}}_M \subseteq (x_1,\ldots, x_t)M:_M \fm^{n_0}.$$ This implies that the kernel of the map $M/(x_1,\ldots, x_t)M \to H^t_{(x_1,\ldots,x_t)}(M)$ is 
$$(x_1,\ldots, x_t)^{\mathrm{lim}}_M/(x_1,\ldots, x_t)M$$ 
that is contained in 
$$((x_1,\ldots, x_t)M:_M \fm^{n_0})/(x_1,\ldots, x_t)M.$$
The map 
$$(x_1,\ldots, x_t)^{\mathrm{lim}}_M/(x_1,\ldots, x_t)M\to 0:_{H^t_{\fm}(M)}\fm^{n_0}$$ is a zero map. Hence,
$$(x_1,\ldots,x_t)_M^{\mathrm{lim}}/(x_1,\ldots,x_t)M \hookrightarrow \bigoplus_{i=0}^{t-1} H^i_{\mathfrak{m}}(M)^{\binom{t}{i}}.$$
On the other hand, by Lemma \ref{lemma qlim} we have 
\[
\ell_R\!\left(
(x_1,\ldots,x_{t})^{\mathrm{lim}}_M /(x_1,\ldots,x_{t})M \right) = \sum_{i=0}^{t-1}\binom{t}{i}\,\ell_R\!\left(H_{\mathfrak m}^i(M)\right).\]
Combining the above observations, we obtain the following 
$$(x_1,\ldots,x_t)_M^{\mathrm{lim}}/(x_1,\ldots,x_t)M \cong \bigoplus_{i=0}^{t-1} H^i_{\mathfrak{m}}(M)^{\binom{t}{i}}.$$
\end{proof}
Applying the previous lemma to $R$, we obtain the following.
\begin{lemma} \label{qlimforring}
Let $(R,\fm)$ be a local ring of positive finiteness dimension $t>0$. Let $n_0$ be a positive integer such that $\fm^{n_0}H_{\fm}^j(R)=0$ for all $j<t$. Then for every filter regular sequence $x_1,\ldots,x_t$ of $R$ contained in $\fm^{2n_0} $, we have
$$ (x_1,\ldots,x_t)^{\mathrm{lim}}/(x_1,\ldots,x_t) \cong \bigoplus_{i=0}^{t-1} H^i_{\mathfrak{m}}(R)^{\binom{t}{i}}.$$ 
If we further assume that $R$ has prime characteristic $p$, then the above isomorphism is compatible with the Frobenius action. Hence
$$ ((x_1,\ldots,x_t)^{\mathrm{lim}}\cap (x_1,\ldots,x_t)^F) /(x_1,\ldots,x_t) \cong \bigoplus_{i=0}^{t-1} {0 ^F_{H^i_{\mathfrak{m}}(R)}} ^{\binom{t}{i}}.$$   
\end{lemma}
Let $I$ be an arbitrary ideal of $R$, and $e$ an integer. Notice that $(I^F)^{[p^e]}\subseteq (I^{[p^e]})^F$, so $\mathrm{Fte}(I)\leq \mathrm{Fte}(I^{[p^e]})+e$.
\begin{theorem} \label{mainthm}
 Let $(R,\fm)$ be a local ring of prime characteristic $p$ and of positive finiteness dimension $t>0$. Let $n_0$ be a positive integer such that $\fm^{n_0} H^i_{\fm}(R) = 0$ for all $i < t$. Then for every filter regular sequence $x_1,\ldots,x_t$ of $R$ we have
$$\mathrm{Fte}(x_1,\ldots,x_t) \le \lceil  \log_p(2n_0)\rceil + \mathrm{HSL}(R),$$
where $\lceil x\rceil$ is the smallest integer that is greater than or equal to $x$. 
\end{theorem}
\begin{proof}
Set $I=(x_1,\ldots,x_t)$.\\[0.15cm] 
Consider an exact sequence $$0\to (I^{\mathrm{lim}}\cap I^F)/I\to I^F/I\to I^F/(I^{\mathrm{lim}}\cap I^F) \to 0.$$
  Moreover, $I^F/(I^{\mathrm{lim}}\cap I^F )\cong (I^{\mathrm{lim}}+I^F)/I^{\mathrm{lim}} \hookrightarrow 0^F_{H^t_{\fm}(R)}$ by \cite[Lemma 3.3]{GMMQS26}. 
 If the filter regular sequence $x_1, \ldots, x_t \in  \fm^{2n_0}$, by Lemma \ref{qlimforring} we have
 $$ (I^{\mathrm{lim}}\cap I^F) /I \cong \bigoplus_{i=0}^{t-1} {0 ^F_{H^i_{\mathfrak{m}}(R)}}^{\binom{t}{i}} .$$
 Applying $F^{\mathrm{HSL}(R)}(-)$ to the above equivalences, we see that $F^{\mathrm{HSL}(R)}(0^F_{H^i_{\fm}(R)})=0$ for all $i\ge 0$ and thus $((I^{\mathrm{lim}}\cap I^F)/I)^{[p^{\mathrm{HSL}(R)}]}= (I^F/(I^{\mathrm{lim}}\cap I^F))^{[p^{\mathrm{HSL}(R)}]}=0$. In other words,
 $$I^{[p^{\mathrm{HSL}(R)}]}=(I^{\mathrm{lim}}\cap I^F)^{[p^{\mathrm{HSL}(R)}]} = (I^F)^{[p^{\mathrm{HSL}(R)}]}.$$
 
 Hence, $\mathrm{Fte}(I)\leq \mathrm{HSL}(R)$ for every  $x_1, \ldots, x_t \in  \fm^{2n_0}$. 
Taking any filter regular sequence $x_1, \ldots, x_t$, set $e_0=\lceil  \log_p(2n_0)\rceil$, then $I^{[p^{e_0}]} \subseteq \fm^{2n_0} $ and thus  $\mathrm{Fte}(I)\leq \mathrm{Fte}(I^{[p^{e_0}]})+e_0$.
Combining all the above observations, we complete the proof.   
\end{proof}
Applying Theorem \ref{mainthm} to generalized Cohen--Macaulay local rings, we obtain the following corollary.
\begin{corollary}
 Let $(R,\fm)$ be a generalized Cohen-Macaulay local ring of prime characteristic $p$. Let $n_0$ be a positive integer such that $\fm^{n_0} H^i_{\fm}(R) = 0$ for all $i < \mathrm{dim}(R)$. Then we have
$$\mathrm{Fte}(R) \le \lceil  \log_p(2n_0)\rceil + \mathrm{HSL}(R),$$
where $\lceil x\rceil$ is the smallest integer that is greater than or equal to $x$.     
\end{corollary}
Together with the inequality $\mathrm{HSL}(R) \le \mathrm{Fte}(R)$ proved in \cite{HQ19}, our result shows that the obtained upper bound is essentially sharp.


\end{document}